\documentclass[11pt,a4paper]{article}

\usepackage{inputenc}
\usepackage{amsmath}
\usepackage{bm}
\usepackage{bbold}
\usepackage{amsthm}

\usepackage{hyperref}

\title{Algebraic Solutions to Multidimensional Minimax Location Problems with Chebyshev Distance\thanks{Recent Researches in Applied and Computational Mathematics: Intern. Conf. on Applied and Computational Mathematics (ICACM’11), WSEAS Press, 2011, pp.~157--162.}
} 

\author{Nikolai Krivulin\thanks{Faculty of Mathematics and Mechanics, St.~Petersburg State University, 28 Universitetsky Ave., St.~Petersburg, 198504, Russia, 
nkk@math.spbu.ru.} \thanks{The work was partially supported by the Russian Foundation for Basic Research under Grant \#09-01-00808.}
}

\date{}

\newtheorem{theorem}{Theorem}
\newtheorem{lemma}[theorem]{Lemma}
\newtheorem{corollary}[theorem]{Corollary}

\begin{document}

\maketitle

\begin{abstract}
Multidimensional minimax single facility location problems with Chebyshev distance are examined within the framework of idempotent algebra. A new algebraic solution based on an extremal property of the eigenvalues of irreducible matrices is given. The solution reduces both unconstrained and constrained location problems to evaluation of the eigenvalue and eigenvectors of an appropriate matrix.
\\

\end{abstract}

\section{Introduction}

Location problems form one of the classical research domains in optimization that has its origin dating back to XVIIth century. Over many years a large body of research on this topic contributed to the development in various areas including integer programming, combinatorial and graph optimization. Among other solution approaches to location problems are models and methods of idempotent algebra \cite{Baccelli1993Synchronization,Cuninghame-Green1994Minimax,Kolokoltsov97Idempotent,Golan2003Semirings,Heidergott2006Max-plus,Butkovic2010Maxlinear}, which find expanding applications in the analysis of actual problems in engineering, manufacturing, information technology, and other fields. Specifically, an algebraic solution to a one-dimensional location problem on a graph is proposed in \cite{Cuninghame-Green1991Minimax,Cuninghame-Green1994Minimax}. A constrained location problem and its representation in terms of idempotent algebra are examined in \cite{Zimmermann2003Disjunctive,Tharwat2010Oneclass}.

In this paper, we consider a multidimensional minimax single facility location problem with Chebyshev distance, and show how the problem can be solved based on new results of the spectral theory in idempotent algebra. The aim of the paper is twofold: first, to give a new algebraic solution to the location problem, and second, to extend the area of application of idempotent algebra. The rest of the paper is as follows. We begin with an overview of preliminary algebraic definitions and results. Specifically, we present an extremal property of the eigenvalue of irreducible matrices. Furthermore, we examine an unconstrained minimax location problem and represent it in the terms of idempotent algebra. A new solution is given that reduces the problem to evaluation of the eigenvalue and eigenvectors of an irreducible matrix. Finally, the solution is extended to solve a constrained location problem.

\section{Preliminary Results}

We start with a brief overview of definitions, notation and preliminary results of idempotent algebra that underlie the solution approach developed in subsequent sections. Further details can be found in \cite{Baccelli1993Synchronization,Kolokoltsov97Idempotent,Cuninghame-Green1994Minimax,Golan2003Semirings,Heidergott2006Max-plus,Butkovic2010Maxlinear}.

\subsection{Idempotent Semifield}

Let $\mathbb{X}$ be a set with two operations, addition $\oplus$ and multiplication $\otimes$, and their respective neutral elements, $\mathbb{0}$ and $\mathbb{1}$. We suppose that  $(\mathbb{X},\mathbb{0},\mathbb{1},\oplus,\otimes)$ is a commutative semiring where the addition is idempotent and the multiplication is invertible. Since the nonzero elements of the semiring form a group under multiplication, the semiring is usually referred to as idempotent semifield. 

The integer power is defined in the ordinary way. Let us put $\mathbb{X}_{+}=\mathbb{X}\setminus\{\mathbb{0}\}$. For any $x\in\mathbb{X}_{+}$ and integer $p>0$, we have $x^{0}=\mathbb{1}$, $x^{p}=x^{p-1}\otimes x=x\otimes x^{p-1}$, $x^{-p}=(x^{-1})^{p}$, and $\mathbb{0}^{p}=\mathbb{0}$.

We assume that the integer power can naturally be extended to the case of rational and real exponents.

In what follows, we drop, as is customary, the multiplication sign $\otimes$. The power notation is used in the sense of idempotent algebra.

The idempotent addition allows one to define a relation of partial order $\leq$ such that $x\leq y$ if and only if $x\oplus y=y$. From the definition it follows that $x\leq x\oplus y$ and $y\leq x\oplus y$, as well as that the addition and multiplication are both isotonic. Below the relation symbols are thought of as referring to this partial order.

It is easy to verify that the binomial identity now takes the form $(x\oplus y)^{\alpha}=x^{\alpha}\oplus y^{\alpha}$ for all $\alpha\geq0$.

As an example, one can consider the idempotent semifield of real numbers
$$
\mathbb{R}_{\max,+}
=
(\mathbb{R}\cup\{-\infty\},-\infty,0,\max,+).
$$

In the semifield $\mathbb{R}_{\max,+}$, there are the null and identity elements defined as $\mathbb{0}=-\infty$ and $\mathbb{1}=0$. For each $x\in\mathbb{R}$, there exists its inverse $x^{-1}$ equal to the opposite number $-x$ in conventional arithmetic. For any $x,y\in\mathbb{R}$, the power $x^{y}$ coincides with the arithmetic product $xy$.

\subsection{Vectors and Matrices}

Vector and matrix operations are routinely introduced based on the scalar addition and multiplication defined on $\mathbb{X}$. Consider the Cartesian power $\mathbb{X}^{n}$ with its elements represented as column vectors. For any two vectors $\bm{x}=(x_{i})$ and $\bm{y}=(y_{i})$, and a scalar $c\in\mathbb{X}$, vector addition and multiplication by scalars follow the rules
$$
\{\bm{x}\oplus\bm{y}\}_{i}
=
x_{i}\oplus y_{i},
\qquad
\{c\bm{x}\}_{i}
=
cx_{i}.
$$ 

The set $\mathbb{X}^{n}$ with these operations is a vector semimodule over the idempotent semifield $\mathbb{X}$.

As it usually is, a vector $\bm{y}\in\mathbb{X}^{n}$ is linearly dependent on vectors $\bm{x}_{1},\ldots,\bm{x}_{m}\in\mathbb{X}^{n}$, if there are scalars $c_{1},\ldots,c_{m}\in\mathbb{X}$ such that $\bm{y}=c_{1}\bm{x}_{1}\oplus\cdots\oplus c_{m}\bm{x}_{m}$. In particular, the vector $\bm{y}$ is collinear with $\bm{x}$, if $\bm{y}=c\bm{x}$.

For any column vector $\bm{x}=(x_{i})\in\mathbb{X}_{+}^{n}$, we define a row vector $\bm{x}^{-}=(x_{i}^{-})$ with its elements $x_{i}^{-}=x_{i}^{-1}$. For all $\bm{x},\bm{y}\in\mathbb{X}_{+}^{n}$, the componentwise inequality $\bm{x}\leq\bm{y}$ implies $\bm{x}^{-}\geq\bm{y}^{-}$.

For any conforming matrices $A=(a_{ij})$, $B=(b_{ij})$, and $C=(c_{ij})$, matrix addition and multiplication together with  multiplication by a scalar $c\in\mathbb{X}$ are performed according to the formulas
$$
\{A\oplus B\}_{ij}
=
a_{ij}\oplus b_{ij},
\qquad
\{B C\}_{ij}
=
\bigoplus_{k}b_{ik}c_{kj},
\qquad
\{cA\}_{ij}=ca_{ij}.
$$

A matrix with all zero entries is referred to as zero matrix and denoted by $\mathbb{0}$.

Consider the set of square matrices $\mathbb{X}^{n\times n}$. The matrix that has all diagonal entries equal to $\mathbb{1}$ and off-diagonal entries equal to $\mathbb{0}$ is the identity matrix denoted by $I$.

With respect to matrix addition and multiplication, the set $\mathbb{X}^{n\times n}$ forms idempotent semiring with identity.

For any matrix $A\ne\mathbb{0}$ and an integer $p>0$, we have $A^{0}=I$ and $A^{p}=A^{p-1}A=AA^{p-1}$.

The trace of the matrix $A=(a_{ij})$ is defined as $\mathop\mathrm{tr}A=a_{11}\oplus\cdots\oplus a_{nn}$.

A matrix is irreducible if and only if it cannot be put in a block triangular form by simultaneous permutations of rows and columns. Otherwise the matrix is reducible.

\subsection{Eigenvalues and Eigenvectors of Matrices}

A scalar $\lambda$ is an eigenvalue of a square matrix $A\in\mathbb{X}^{n\times n}$ if there exists a vector $\bm{x}\in\mathbb{X}^{n}\setminus\{\mathbb{0}\}$ such that $A\bm{x}=\lambda\bm{x}$. Any vector $\bm{x}\ne\mathbb{0}$ that satisfies the equation is an eigenvector of $A$, corresponding to $\lambda$.

If the matrix $A$ is irreducible, then it has only one eigenvalue given by
\begin{equation}
\lambda
=
\bigoplus_{m=1}^{n}\mathop\mathrm{tr}\nolimits^{1/m}(A^{m}).
\label{E-lambda}
\end{equation}

The corresponding eigenvectors of $A$ have no zero entries and are found as follows. First we evaluate the matrix
$$
A^{\times}
=
\lambda^{-1}A\oplus\cdots\oplus(\lambda^{-1}A)^{n}.
$$

Let $\bm{a}_{i}^{\times}$ and $a_{ii}^{\times}$ be respective column $i$ and diagonal entry $(i,i)$ of $A^{\times}$. Consider the subset of columns $\bm{a}_{i}^{\times}$ such that $a_{ii}^{\times}=\mathbb{1}$, $i=1,\ldots,n$. In the subset, find those columns that are linearly independent of the others, and take them to form a matrix $A^{+}$.

The set of all eigenvectors of $A$ corresponding to $\lambda$ (together with zero vector) coincides with the linear span of the columns of $A^{\times}$, whereas each vector takes the form
$$
\bm{x}
=
A^{+}\bm{v},
$$
where $\bm{v}$ is a nonzero vector of appropriate size.

\section{An Extremal Property of Eigenvalues}

Suppose $A\in\mathbb{X}^{n\times n}$ is an irreducible matrix with an eigenvalue $\lambda$. Consider a function $\varphi(\bm{x})=\bm{x}^{-}A\bm{x}$ defined on $\mathbb{X}_{+}^{n}$. It has been shown in \cite{Krivulin2005Evaluation,Krivulin2006Eigenvalues} that $\varphi(\bm{x})$ has $\lambda$ as its minimum, which is attained at any eigenvector of $A$.

Now we improve the above result by extending the set of vectors that provide for the minimum of $\varphi(\bm{x})$.
\begin{lemma}\label{L-mxmAx}
Let $A=(a_{ij})\in\mathbb{X}^{n\times n}$ be an irreducible matrix with an eigenvalue $\lambda$. Suppose $\bm{u}=(u_{i})$ and $\bm{v}=(v_{i})$ are eigenvectors of the respective  matrices $A$ and $A^{T}$. Then it holds that
$$
\min_{\bm{x}\in\mathbb{X}_{+}^{n}}\bm{x}^{-} A\bm{x}
=
\lambda,
$$
where the minimum is attained at any vector $\bm{x}=(u_{1}^{\alpha}v_{1}^{\alpha-1},\ldots,u_{n}^{\alpha}v_{n}^{\alpha-1})^{T}$ for all $\alpha$ such that $0\leq\alpha\leq1$.
\end{lemma}
\begin{proof}
It is easy to verify as in \cite{Krivulin2006Eigenvalues} that any vector $\bm{x}$ with nonzero elements satisfies the inequality $\bm{x}^{-}A\bm{x}\geq\lambda$. Indeed, since $\bm{x}\bm{u}^{-}\geq(\bm{x}^{-}\bm{u})^{-1}I$, we have $\bm{x}^{-} A\bm{x}=\bm{x}^{-} A\bm{x}\bm{u}^{-}\bm{u}\geq\bm{x}^{-} A\bm{u}(\bm{x}^{-}\bm{u})^{-1}=\lambda\bm{x}^{-}\bm{u}(\bm{x}^{-}\bm{u})^{-1}=\lambda$.

It remains to present a vector $\bm{x}$ that turns the inequality into an equality. With $\bm{x}=\bm{u}$ we immediately have $\bm{x}^{-}A\bm{x}=\lambda\bm{u}^{-}\bm{u}=\lambda$. Similarly, if $\bm{x}=(\bm{v}^{-})^{T}$, then $\bm{x}^{-}A\bm{x}=\bm{x}^{T}A^{T}(\bm{x}^{-})^{T}=\bm{v}^{-}A^{T}\bm{v}=\lambda\bm{v}^{-}\bm{v}=\lambda$.

Let us now take any vector $\bm{x}=(x_{i})$ with elements $x_{i}=u_{i}^{\alpha}v_{i}^{\alpha-1}$, where $\alpha$ is a real number such that $0\leq\alpha\leq1$. With the following calculations
\begin{multline*}
\lambda
=
(\bm{u}^{-}A\bm{u})^{\alpha}
(\bm{v}^{-}A^{T}\bm{v})^{1-\alpha}
=
\bigoplus_{i=1}^{n}\bigoplus_{j=1}^{n}u_{i}^{-\alpha}a_{ij}^{\alpha}u_{j}^{\alpha}
\bigoplus_{k=1}^{n}\bigoplus_{l=1}^{n}v_{k}^{\alpha-1}a_{lk}^{1-\alpha}v_{l}^{1-\alpha}
\\
\geq
\bigoplus_{i=1}^{n}\bigoplus_{j=1}^{n}u_{i}^{-\alpha}v_{i}^{1-\alpha}a_{ij}u_{j}^{\alpha}v_{j}^{\alpha-1}
=
\bigoplus_{i=1}^{n}\bigoplus_{j=1}^{n}x_{i}^{-1}a_{ij}x_{j}
=
\bm{x}^{-}A\bm{x},
\end{multline*}
we arrive at the inequality $\bm{x}^{-}A\bm{x}\leq\lambda$. Since the opposite inequality is always valid, we conclude that the vector $\bm{x}$ satisfies the condition $\bm{x}^{-}A\bm{x}=\lambda$.
\end{proof}

\section{The Unconstrained Location Problem}

In this section we consider a minimax single facility location problem with Chebyshev distance. For any two vectors $\bm{r}=(r_{1},\ldots,r_{n})^{T}$ and $\bm{s}=(s_{1},\ldots,s_{n})^{T}$ in $\mathbb{R}^{n}$, the Chebyshev distance ($L_{\infty}$ or maximum metric) is calculated as
\begin{equation}
\rho(\bm{r},\bm{s})
=
\max_{1\leq i\leq n}|r_{i}-s_{i}|.
\label{M-Chebyshev}
\end{equation}

Given $m\geq2$ vectors $\bm{r}_{i}=(r_{1i},\ldots,r_{ni})^{T}\in\mathbb{R}^{n}$ and constants $w_{i}\in\mathbb{R}$, $i=1,\ldots,m$, the problem under consideration is to find a vector $\bm{x}=(x_{1},\ldots,x_{n})^{T}$ so as to provide
\begin{equation}
\min_{\bm{x}\in\mathbb{R}^{n}}\max_{1\leq i\leq m}(\rho(\bm{r}_{i},\bm{x})+w_{i}).
\label{P-Chebyshev}
\end{equation}

%Note that such problems are known as unweighted Weber-Rawls (or Rawls) problems \cite{Hansen81Constrained}. Following the terminology of \cite{Elzinga72Geometrical}, the problem can also be referred to as the Chebyshev Messenger Boy Problem.

It is not difficult to solve the problem by using geometric arguments (see, eg, \cite{Sule2001Logistics,Moradi2009Single}). Below we give a new algebraic solution that is based on representation of the problem in terms of the idempotent semifield $\mathbb{R}_{\max,+}$ and application of the result from the previous section. First we rewrite \eqref{M-Chebyshev} as follows
$$
\rho(\bm{r},\bm{s})
=
\bm{s}^{-}\bm{r}\oplus\bm{r}^{-}\bm{s}.
$$

Denote the objective function of the problem by $\varphi(\bm{x})$. With the vectors
$$
\bm{p}
=
w_{1}\bm{r}_{1}\oplus\cdots\oplus w_{m}\bm{r}_{m},
\qquad
\bm{q}^{-}
=
w_{1}\bm{r}_{1}^{-}\oplus\cdots\oplus w_{m}\bm{r}_{m}^{-},
$$
we can write
$$
\varphi(\bm{x})
=
\bigoplus_{i=1}^{m}w_{i}\rho(\bm{r}_{i},\bm{x})
=
\bigoplus_{i=1}^{m}w_{i}(\bm{x}^{-}\bm{r}_{i}\oplus\bm{r}_{i}^{-}\bm{x})
=
\bm{x}^{-}\bm{p}
\oplus
\bm{q}^{-}\bm{x},
$$
and then represent problem \eqref{P-Chebyshev} as
\begin{equation}\label{P-Chebyshev1}
\min_{\bm{x}\in\mathbb{R}^{n}}\varphi(\bm{x})
=
\min_{\bm{x}\in\mathbb{R}^{n}}(\bm{x}^{-}\bm{p}\oplus\bm{q}^{-}\bm{x}).
\end{equation}

Furthermore, we introduce a vector $\bm{y}$ and a matrix $A$ of order $n+1$ as follows
$$
\bm{y}
=
\left(
\begin{array}{c}
\mathbb{1} \\
\bm{x}
\end{array}
\right),
\qquad
A
=
\left(
\begin{array}{cc}
\mathbb{0} & \bm{q}^{-} \\
\bm{p} & \mathbb{0}
\end{array}
\right).
$$

Since we now have $\varphi(\bm{x})=\bm{x}^{-}\bm{p}\oplus\bm{q}^{-}\bm{x}=\bm{y}^{-}A\bm{y}$, problem \eqref{P-Chebyshev1} reduces to that of the form
\begin{equation}\label{P-Chebyshev2}
\min_{\bm{y}\in\mathbb{R}^{n+1}}\bm{y}^{-}A\bm{y}.
\end{equation}

Note that the vectors $\bm{y}\in\mathbb{R}^{n+1}$ that solve \eqref{P-Chebyshev2} do not always have an appropriate form to give a solutions to \eqref{P-Chebyshev1}. Specifically, to be consistent to \eqref{P-Chebyshev1}, the vector $\bm{y}$ must have the first element equal to $\mathbb{1}$.

\section{Algebraic Solution of the Unconstrained Problem}

Consider problem \eqref{P-Chebyshev2}, and note that the matrix $A$ is irreducible. It follows from Lemma~\ref{L-mxmAx} that the minimum in \eqref{P-Chebyshev2} is equal to the eigenvalue $\lambda$ of the matrix $A$. For all $k=1,2,\ldots$ we have
$$
A^{2k-1}
=
(\bm{q}^{-}\bm{p})^{k-1}
\left(
\begin{array}{cc}
\mathbb{0} & \bm{q}^{-} \\
\bm{p} & \mathbb{0}
\end{array}
\right),
\qquad
A^{2k}
=
(\bm{q}^{-}\bm{p})^{k-1}
\left(
\begin{array}{cc}
\bm{q}^{-}\bm{p} & \mathbb{0} \\
\mathbb{0} & \bm{p}\bm{q}^{-}
\end{array}
\right),
$$
and therefore, $\mathop\mathrm{tr}(A^{2k-1})=\mathbb{0}$, $\mathop\mathrm{tr}(A^{2k})=(\bm{q}^{-}\bm{p})^{k}$. Finally, application of \eqref{E-lambda} gives
$$
\lambda
=
(\bm{q}^{-}\bm{p})^{1/2}.
$$

To find vectors that produce the minimum in \eqref{P-Chebyshev2}, we need to derive the eigenvectors of the matrices $A$ and $A^{T}$. Note that $A^{T}$ is obtained from $A$ by replacement of $\bm{p}$ with $(\bm{q}^{-})^{T}$ and $\bm{q}^{-}$ with $\bm{p}^{T}$. Therefore, it will suffice to find the eigenvectors for $A$, and then turn them into the eigenvectors for $A^{T}$ by the above replacement.

To get the eigenvectors of $A$, we first find the matrix $A^{\times}$. Since for any $k=1,2,\ldots$ it holds that
\begin{align*}
(\lambda^{-1}A)^{2k-1}
&=
(\bm{q}^{-}\bm{p})^{-1/2}
\left(
\begin{array}{cc}
\mathbb{0} & \bm{q}^{-} \\
\bm{p} & \mathbb{0}
\end{array}
\right),
\\
(\lambda^{-1}A)^{2k}
&=
(\bm{q}^{-}\bm{p})^{-1}
\left(
\begin{array}{cc}
\bm{q}^{-}\bm{p} & \mathbb{0} \\
\mathbb{0} & \bm{p}\bm{q}^{-}
\end{array}
\right),
\end{align*}
we arrive at the matrix $A^{\times}$ in the form
$$
A^{\times}
=
\lambda^{-1}A\oplus\cdots\oplus(\lambda^{-1}A)^{n+1}
=
\left(
\begin{array}{cc}
\mathbb{1} & (\bm{q}^{-}\bm{p})^{-1/2}\bm{q}^{-} \\
(\bm{q}^{-}\bm{p})^{-1/2}\bm{p} & (\bm{q}^{-}\bm{p})^{-1}\bm{p}\bm{q}^{-}
\end{array}
\right).
$$

It is not difficult to verify that in the matrix $A^{\times}$, any column that has $\mathbb{1}$ on the diagonal is collinear with the first column. Indeed, suppose that the submatrix $(\bm{q}^{-}\bm{p})^{-1}\bm{p}\bm{q}^{-}$ has a diagonal element equal to $\mathbb{1}$, say the element in its first column (that corresponds to the second column of $A^{\times}$). In this case, we have $\bm{q}^{-}\bm{p}=q_{1}^{-1}p_{1}$, whereas the matrix $A^{\times}$ takes the form
$$
A^{\times}
=
\left(
\begin{array}{cc}
\mathbb{1} & q_{1}^{1/2}p_{1}^{-1/2}\bm{q}^{-} \\
q_{1}^{1/2}p_{1}^{-1/2}\bm{p} & q_{1}p_{1}^{-1}\bm{p}\bm{q}^{-}
\end{array}
\right)
=
\left(
\begin{array}{ccc}
\mathbb{1} & q_{1}^{-1/2}p_{1}^{-1/2} & \ldots\\
q_{1}^{1/2}p_{1}^{-1/2}\bm{p} & p_{1}^{-1}\bm{p} & \ldots
\end{array}
\right),
$$
where the second column proves to be collinear with the first.

Let us construct a matrix $A^{+}$ that includes such columns of $A^{\times}$ that have the diagonal element equal to $\mathbb{1}$. Since all these columns are collinear with the first one, they can be omitted. With the matrix $A^{+}$ formed from the first column of $A^{\times}$, we finally represent any eigenvector of $A$ as
$$
\bm{u}
=
\left(
\begin{array}{c}
\mathbb{1} \\
(\bm{q}^{-}\bm{p})^{-1/2}\bm{p}
\end{array}
\right)s,
\qquad
s\in\mathbb{R}.
$$
 
By replacing $\bm{p}$ with $(\bm{q}^{-})^{T}$ and $\bm{q}^{-}$ with $\bm{p}^{T}$, we get the eigenvectors of $A^{T}$
$$
\bm{v}
=
\left(
\begin{array}{c}
\mathbb{1} \\
(\bm{q}^{-}\bm{p})^{-1/2}(\bm{q}^{-})^{T}
\end{array}
\right)t,
\qquad
t\in\mathbb{R}.
$$

It follows from Lemma~\ref{L-mxmAx} that the solution of \eqref{P-Chebyshev1} takes the form
$$
\bm{y}
=
\left(
\begin{array}{c}
u_{0}^{\alpha}v_{0}^{\alpha-1} \\
u_{1}^{\alpha}v_{1}^{\alpha-1} \\
\vdots \\
u_{n+1}^{\alpha}v_{n+1}^{\alpha-1}
\end{array}
\right)
=
\left(
\begin{array}{c}
\mathbb{1} \\
(\bm{q}^{-}\bm{p})^{1/2-\alpha}p_{1}^{\alpha}q_{1}^{1-\alpha} \\
\vdots \\
(\bm{q}^{-}\bm{p})^{1/2-\alpha}p_{n}^{\alpha}q_{n}^{1-\alpha}
\end{array}
\right)s^{\alpha}t^{\alpha-1},
\qquad
s,t\in\mathbb{R}.
$$

With the condition that the first element of $\bm{y}$ must be equal to $\mathbb{1}$, we have to set $s^{\alpha}t^{\alpha-1}=\mathbb{1}$. Going back to problem \eqref{P-Chebyshev1}, we arrive at the following result.
\begin{lemma}\label{L-Chebyshev}
The minimum in problem \eqref{P-Chebyshev1} is given by $\lambda=(\bm{q}^{-}\bm{p})^{1/2}$, and it is attained at any vector
$$
\bm{x}
=
\lambda^{1-2\alpha}
\left(
\begin{array}{c}
p_{1}^{\alpha}q_{1}^{1-\alpha} \\
\vdots \\
p_{n}^{\alpha}q_{n}^{1-\alpha}
\end{array}
\right),
\qquad
0\leq\alpha\leq1.
$$
\end{lemma}

With the usual notation, we can reformulate the statement of Lemma~\ref{L-Chebyshev} as follows.
\begin{corollary}
The minimum in \eqref{P-Chebyshev} is given by $\lambda=\max(p_{1}-q_{1},\ldots,p_{n}-q_{n})/2$, and it is attained at any vector
$$
\bm{x}
=
\alpha
\left(
\begin{array}{c}
p_{1}-\lambda \\
\vdots \\
p_{n}-\lambda
\end{array}
\right)
+
(1-\alpha)
\left(
\begin{array}{c}
q_{1}+\lambda \\
\vdots \\
q_{n}+\lambda
\end{array}
\right),
\qquad
0\leq\alpha\leq1,
$$
where $p_{i}=\max(r_{i1}+w_{1},\ldots,r_{im}+w_{m})$, $q_{i}=\min(r_{i1}-w_{1},\ldots,r_{im}-w_{m})$ for each $i=1,\ldots,n$.
\end{corollary}

\section{A Constrained Location Problem}

Suppose that there is a set $S\in\mathbb{R}^{n}$ given to specify feasible location points. Consider a constrained location problem that is represented in terms of the semifield $\mathbb{R}_{\max,+}$ in the form
\begin{equation}
\min_{\bm{x}\in S}\bigoplus_{i=1}^{m}w_{i}\rho(\bm{r}_{i},\bm{x}).
\label{P-ChebyshevConstrained}
\end{equation}

To solve the problem we put it in the form of \eqref{P-Chebyshev1} by including the area constraints into the objective function of an unconstrained problem. First we slightly transform problem \eqref{P-Chebyshev1} to enable accommodation of the constraints in a natural way. With the notation
\begin{gather*}
\bm{p}_{0}
=
w_{1}\bm{r}_{1}\oplus\cdots\oplus w_{m}\bm{r}_{m},
\qquad
\bm{q}_{0}^{-}
=
w_{1}\bm{r}_{1}^{-}\oplus\cdots\oplus w_{m}\bm{r}_{m}^{-},
\\
\lambda_{0}
=
(\bm{q}_{0}^{-}\bm{p}_{0})^{1/2},
\qquad
\varphi_{0}(\bm{x})
=
\lambda_{0}^{-1}(\bm{x}^{-}\bm{p}_{0}\oplus\bm{q}_{0}^{-}\bm{x}),
\end{gather*}
we turn to the problem
$$
\min_{\bm{x}\in\mathbb{R}^{n}}\varphi_{0}(\bm{x}).
$$

It follows from Lemma~\ref{L-Chebyshev} that the last problem has its minimum equal to $\mathbb{1}=0$, whereas its solution set obviously coincides with that of \eqref{P-Chebyshev1}.

Suppose that there are constraints on the maximum distance from the facility location point to each given points that determine the feasible location set in the form $S=\{\bm{x}\in\mathbb{R}^{n}|\rho(\bm{r}_{i},\bm{x})\leq d_{i}, i=1,\ldots,m\}$.

For each $i=1,\ldots,m$, the inequality $\bm{x}^{-}\bm{r}_{i}\oplus\bm{r}_{i}^{-}\bm{x}=\rho(\bm{r}_{i},\bm{x})\leq d_{i}$ can be rewritten in an equivalent form as $d_{i}^{-1}\bm{x}^{-}\bm{r}_{i}\oplus d_{i}^{-1}\bm{r}_{i}^{-}\bm{x}\leq\mathbb{1}$, whereas all constraints can be replaced with one inequality
$$
\bm{x}^{-}(d_{1}^{-1}\bm{r}_{1}\oplus\cdots\oplus d_{m}^{-1}\bm{r}_{m})\oplus(d_{1}^{-1}\bm{r}_{1}^{-}\oplus\cdots\oplus d_{m}^{-1}\bm{r}_{m}^{-})\bm{x}
\leq
\mathbb{1}.
$$

We introduce the notation
\begin{gather*}
\bm{p}_{1}
=
d_{1}^{-1}\bm{r}_{1}\oplus\cdots\oplus d_{m}^{-1}\bm{r}_{m},
\qquad
\bm{q}_{1}^{-}
=
d_{1}^{-1}\bm{r}_{1}^{-}\oplus\cdots\oplus d_{m}^{-1}\bm{r}_{m}^{-},
\\
\varphi_{1}(\bm{x})
=
\bm{x}^{-}\bm{p}_{1}\oplus\bm{q}_{1}^{-}\bm{x},
\end{gather*}
and note that $\varphi_{1}(\bm{x})\leq\mathbb{1}$ when and only when the above constraints are satisfied.

Furthermore, we put
\begin{gather*}
\bm{p}
=
\lambda_{0}^{-1}\bm{p}_{0}\oplus\bm{p}_{1},
\qquad
\bm{q}^{-}
=
\lambda_{0}^{-1}\bm{q}_{0}^{-}\oplus\bm{q}_{1}^{-},
\\
\varphi(\bm{x})
=
\varphi_{0}(\bm{x})
\oplus
\varphi_{1}(\bm{x})
=
\bm{x}^{-}\bm{p}\oplus\bm{q}^{-}\bm{x}.
\end{gather*}

Now we can replace the original problem \eqref{P-ChebyshevConstrained} by an unconstrained problem with the objective function $\varphi(\bm{x})$ that has the form of problem \eqref{P-Chebyshev1} with the solution given by Lemma~\ref{L-Chebyshev}. It is clear that both problems give the same solution set provided that a solution exists. At the same time, the new problem allows one to get approximate solutions in the case when the original problem does not have a solution.

\bibliographystyle{utphys}

\bibliography{Algebraic_solutions_to_multidimensional_minimax_location_problems_with_Chebyshev_distance}

\providecommand{\href}[2]{#2}\begingroup\raggedright\begin{thebibliography}{10}

\bibitem{Baccelli1993Synchronization}
F.~L. Baccelli, G.~Cohen, G.~J. Olsder, and J.-P. Quadrat, {\em Synchronization
  and Linearity: An Algebra for Discrete Event Systems}.
\newblock Wiley Series in Probability and Statistics. Wiley, Chichester, 1993.
\newblock \url{http://www-rocq.inria.fr/metalau/cohen/documents/BCOQ-book.pdf}.

\bibitem{Cuninghame-Green1994Minimax}
R.~A. Cuninghame-Green,
  \href{http://dx.doi.org/10.1016/S1076-5670(08)70083-1}{``Minimax algebra and
  applications,''} in {\em Advances in Imaging and Electron Physics, Vol. 90},
  P.~W. Hawkes, ed., pp.~1--121.
\newblock Academic Press, San Diego, 1994.

\bibitem{Kolokoltsov97Idempotent}
V.~N. Kolokoltsov and V.~P. Maslov, {\em Idempotent Analysis and Its
  Applications}, vol.~401 of {\em Mathematics and Its Applications}.
\newblock Kluwer Academic Publishers, Dordrecht, 1997.

\bibitem{Golan2003Semirings}
J.~S. Golan, {\em Semirings and Affine Equations Over Them: Theory and
  Applications}, vol.~556 of {\em Mathematics and Its Applications}.
\newblock Springer, New York, 2003.

\bibitem{Heidergott2006Max-plus}
B.~Heidergott, G.~J. Olsder, and J.~van~der Woude, {\em Max-plus at Work:
  Modeling and Analysis of Synchronized Systems}.
\newblock Princeton Series in Applied Mathematics. Princeton University Press,
  Princeton, 2006.

\bibitem{Butkovic2010Maxlinear}
P.~Butkovi\v{c}, \href{http://dx.doi.org/10.1007/978-1-84996-299-5}{{\em
  Max-linear Systems: Theory and Algorithms}}.
\newblock Springer Monographs in Mathematics. Springer, London, 2010.

\bibitem{Cuninghame-Green1991Minimax}
R.~Cuninghame-Green, ``Minimax algebra and applications,''
  \href{http://dx.doi.org/10.1016/0165-0114(91)90130-I}{{\em Fuzzy Sets and
  Systems} {\bfseries 41} no.~3, (1991) 251--267}.

\bibitem{Zimmermann2003Disjunctive}
K.~Zimmermann, ``Disjunctive optimization, max-separable problems and extremal
  algebras,'' \href{http://dx.doi.org/10.1016/S0304-3975(02)00231-1}{{\em
  Theoret. Comput. Sci.} {\bfseries 293} no.~1, (2003) 45--54}.

\bibitem{Tharwat2010Oneclass}
A.~Tharwat and K.~Zimmermann, ``One class of separable optimization problems:
  solution method, application,''
  \href{http://dx.doi.org/10.1080/02331930801954698}{{\em Optimization}
  {\bfseries 59} no.~5, (2010) 619--625}.

\bibitem{Krivulin2005Evaluation}
N.~K. Krivulin, ``Evaluation of bounds on the mean rate of growth of the state
  vector of a linear dynamical stochastic system in idempotent algebra,'' {\em
  Vestnik St. Petersburg Univ. Math.} {\bfseries 38} no.~2, (June, 2005)
  42--51.

\bibitem{Krivulin2006Eigenvalues}
N.~K. Krivulin, ``Eigenvalues and eigenvectors of matrices in idempotent
  algebra,'' {\em Vestnik St. Petersburg Univ. Math.} {\bfseries 39} no.~2,
  (June, 2006) 72--83.

\bibitem{Sule2001Logistics}
D.~R. Sule, {\em Logistics of Facility Location and Allocation}.
\newblock Marcel Dekker, New York, 2001.

\bibitem{Moradi2009Single}
E.~Moradi and M.~Bidkhori,
  \href{http://dx.doi.org/10.1007/978-3-7908-2151-2_3}{``Single facility
  location problem,''} in {\em Facility Location}, R.~Zanjirani~Farahani and
  M.~Hekmatfar, eds., Contributions to Management Science, pp.~37--68.
\newblock Physica-Verlag, Heidelberg, 2009.

\end{thebibliography}\endgroup

\end{document}